\documentclass{article}
\usepackage[margin=1in]{geometry}
\usepackage[utf8]{inputenc}
\usepackage[T1]{fontenc}
\usepackage[english]{babel}

\usepackage{amsmath, amssymb, amsfonts}
\usepackage{mathtools}
\usepackage{dsfont}

\usepackage{tikz-cd}

\usepackage{amsthm}
\usepackage{thmtools}
\usepackage{enumitem}

\usepackage{comment}

\usepackage[colorinlistoftodos,obeyFinal]{todonotes}
\usepackage{xpatch}
\usepackage[explicit]{titlesec}

\usepackage{zref-clever}

\usepackage[hypertexnames=false]{hyperref}

\def\thissectiontitle{}
\def\thissectionnumber{}
\def\thissubsectiontitle{}
\def\thissubsectionnumber{}

\newtoggle{todoSection}
\newtoggle{todoSubsection}

\titleformat{\section}
  {\normalfont\Large\bfseries\sffamily}
  {\thesection}
  {0.5em}
  {\gdef\thissectiontitle{#1}\gdef\thissectionnumber{\thesection}#1}

\titleformat{\subsection}
  {\normalfont\large\bfseries\sffamily}
  {\thesubsection}
  {0.5em}
  {\gdef\thissubsectiontitle{#1}\gdef\thissubsectionnumber{\thesubsection}#1}

  \pretocmd{\section}{\global\toggletrue{todoSection}}{}{}
  \pretocmd{\subsection}{\global\toggletrue{todoSubsection}}{}{}

\AtBeginDocument{%
  \xpretocmd{\todo}{%
    \iftoggle{todoSubsection}{
     \addtocontents{tdo}{\protect\contentsline{subsection}%
        {\protect\numberline{\thissubsectionnumber}{\thissubsectiontitle}}{}{} }
      \global\togglefalse{todoSubsection}
        }{}
    }{}{}%
  }
\AtBeginDocument{%
  \xpretocmd{\todo}{%
    \iftoggle{todoSection}{
     \addtocontents{tdo}{\protect\contentsline{section}%
        {\protect\numberline{\thissectionnumber}{\thissectiontitle}}{}{} }
      \global\togglefalse{todoSection}
        }{}
    }{}{}%
  }

\zcsetup{noabbrev}
\zcsetup{cap}
\zcsetup{hyperref=true}
\zcsetup{nameinlink=false}
\newcommand{\cref}[1]{\zcref{#1}}

\theoremstyle{definition}
\newtheorem{defn}{Definition}[section]

\newtheorem{notation}[defn]{Notation}

\newtheorem{recoll}[defn]{Recollection}

\theoremstyle{plain}
\newtheorem{lem}[defn]{Lemma}
\newtheorem{conj}[defn]{Conjecture}
\newtheorem*{conj*}{Conjecture}
\newtheorem{cor}[defn]{Corollary}
\newtheorem{prop}[defn]{Proposition}
\newtheorem{thm}[defn]{Theorem}
\newtheorem{introthm}{Theorem}

\newtheorem{introcor}[introthm]{Corollary}

\newtheorem{exmpl}[defn]{Example}

\theoremstyle{remark}
\newtheorem{rmk}[defn]{Remark}
\newtheorem*{rmk*}{Remark}

\newcommand{\zcrefname}[3]{%
  \zcRefTypeSetup{#1}{
    Name-sg = #2 ,
    name-sg = #2 ,
    Name-pl = #3 ,
    name-pl = #3 ,
  }
}

\zcrefname{lem}{Lemma}{Lemmas}
\zcrefname{defn}{Definition}{Definitions}
\zcrefname{cor}{Corollary}{Corollaries}
\zcrefname{introcor}{Corollary}{Corollaries}
\zcrefname{prop}{Proposition}{Propositions}
\zcrefname{conj}{Conjecture}{Conjectures}
\zcrefname{thm}{Theorem}{Theorems}
\zcrefname{rmk}{Remark}{Remarks}
\zcrefname{section}{Section}{Sections}
\zcrefname{equation}{Equivalence}{Equivalences}
\zcrefname{constr}{Construction}{Constructions}
\zcrefname{exmpl}{Example}{Examples}
\zcrefname{recoll}{Recollection}{Recollections}
\zcrefname{notation}{Notation}{Notations}

\theoremstyle{definition}

\theoremstyle{plain}

\newcommand{\Z}{{\mathbb{Z}}}
\newcommand{\N}{{\mathbb{N}}}
\newcommand{\Q}{{\mathbb{Q}}}

\newcommand{\C}{{\mathbb{C}}}
\renewcommand{\S}{{\mathbb S}}
\newcommand{\finfld}[1]{\mathbb{F}_{#1}}

\newcommand{\colim}[1]{{\underset{\substack{#1}}{\operatorname{colim}}\,}}
\newcommand{\colimil}[1]{{{\operatorname{colim}}_{#1}\,}}
\renewcommand{\lim}[1]{{\underset{\substack{#1}}{\operatorname{lim}}\,}}

\newcommand{\Map}[3]{{\operatorname{Map}_{#1}({#2},{#3})}}
\newcommand{\map}[3]{{\operatorname{map}_{#1}({#2},{#3})}}
\newcommand{\imap}[3]{{\underline{\smash{\operatorname{map}}}_{#1}({#2},{#3})}}

\newcommand{\Sp}{{\operatorname{Sp}}}
\newcommand{\SH}[1]{\mathcal{SH}(#1)}
\newcommand{\SHet}[1]{\mathcal{SH}_{\et}(#1)}

\newcommand{\Fib}[1]{{\operatorname{fib}\mathopen{}\left(#1\right)\mathclose{}}}

\newcommand{\Cat}[1]{{\mathcal {#1}}}

\newcommand{\ShvTopH}[2]{\operatorname{Shv}^{\operatorname{h}}_{#1}({#2})}

\newcommand{\CAlg}[1]{{\operatorname{CAlg}(#1)}}

\newcommand{\Mod}[2]{\operatorname{Mod}_{#1}(#2)}

\newcommand{\et}{\operatorname{\acute{e}t}}

\newcommand{\Sm}[1]{\operatorname{Sm}_{#1}}

\newcommand{\Gm}{\mathbb{G}_m}

\DeclareMathOperator{\Spec}{Spec}

\newcommand{\sslash}{/\mkern-6mu/}

\newcommand{\cartsymb}{\arrow[dr, phantom,"\scalebox{1.5}{\color{black}$\lrcorner$}", near start, color=black]}

\usepackage{dsfont}

\let\OLDthebibliography\thebibliography
\renewcommand\thebibliography[1]{
  \OLDthebibliography{#1}
  \setlength{\parskip}{0pt}
  \setlength{\itemsep}{0pt plus 0.3ex}
}

\usepackage{tocloft}
\cftsetindents{section}{0em}{2em}
\title{Nilpotence of \texorpdfstring{$\eta$}{eta} in étale motivic spectra}
\author{Klaus Mattis and Swann Tubach}

\begin{document}

\maketitle

\begin{abstract}
    We show that every object of the stable étale motivic homotopy category over any scheme is \texorpdfstring{$\eta$}{eta}-complete. 
    In some cases we show that in fact the fourth power of \texorpdfstring{$\eta$}{eta} is null, whereas the third power of $\eta$ is always nonvanishing,
    similar to the situation in topology. Moreover, we prove an étale version of May's nilpotence conjecture,
    that states that $H\Z \in \Sp$ detects the vanishing of $\mathbf{E}_\infty$-rings. 
    We use this to show a version of Nishida's nilpotence theorem in $\SHet{S}$,
    i.e.\ that any positive degree self map of the unit is nilpotent.
\end{abstract}

\hypersetup{pdfborder=0 0 0}
\tableofcontents
\hypersetup{pdfborder=1 1 1}


\section*{Introduction}

In topology, the Hopf map provides the first example of a nonzero element of a homotopy group of the form 
$\pi_n(S^{n-1})$, it is a fibration $S^3\to S^2$ whose fibers are all isomorphic to $S^1$. A simple definition is as follows:
\[\eta_\mathrm{top}\colon S^3\simeq \C^2\setminus\{0\} \to \C\mathbb{P}^1\simeq S^2\] is the projection, up to homotopy. The (desuspended) image of $\eta_\mathrm{top}$, pointed at any point of $S^3$, in the category of spectra $\Sp$ is a map 
\[\eta_\mathrm{top}\colon \Sigma \S\to \S\] from the suspension of the sphere spectrum $\S$ to the sphere spectrum. 
It provides a generator for the first stable homotopy group of the sphere: $$0\neq\eta_\mathrm{top}\in\pi_1(\S)\simeq \Z/2\Z.$$ 
One may compute the powers of $\eta_\mathrm{top}$, and show that $\eta_\mathrm{top}^4=0$. This is easy, as $\pi_4(\S)=0$,
which can be read of from the $E_2$-page of the Adams spectral sequence. 
It is true, but harder to see, that $\eta_\mathrm{top}^3\neq 0 \in \pi_3(\S)=\Z/24\Z$. 
Indeed, this can be done by computations with Toda brackets, see \cite[Chapter V, Equation 5.5 and Proposition 5.6]{toda1962composition}. 

Motivic homotopy theory aims to imitate the methods of algebraic topology in the world of algebraic geometry. For that matter, many classical results of 
homotopy theory have now a version in algebraic geometry, see e.g.\ \cite{hoyois2015hopkinsmorel,asok2017affine,asok2023freudenthal}. Let $S$ be a scheme. Following Morel and Voevodsky, one considers the stable $\infty$-category $\SH{S}$ of $\mathbb{A}^1$-invariant motivic spectra. As in topology, we have an algebraic Hopf map given by the canonical projection
\[\eta\colon \mathbb{A}^2_S\setminus\{0\} \to \mathbb{P}^1_S\] whose (desuspended) image in $\SH{S}$ is a map 
\[\eta\colon \Gm\to \S\] from the motivic sphere $\Gm$ to the motivic sphere spectrum $\S$. 
As a corollary to a theorem of Morel \cite[Corollary 6.4.5]{morel2004motivic}, one may compute over a perfect field $k$ 
the endomorphisms of the $\eta$-inverted sphere
$\S[\eta^{-1}] \coloneqq \colim{} (\S \xrightarrow{\eta} \Gm^{\otimes -1} \xrightarrow{\eta} \Gm^{\otimes -2} \to \cdots)$: 
we have 
\[\mathrm{End}_{\SH{k}}(\S[\eta^{-1}])\simeq W(k),\] 
where $W(k)$ is the Witt ring of symmetric bilinear forms of $k$. 
In particular, by pulling back to fields, we see that the map $\eta$ is \emph{never nilpotent} in $\SH{S}$, for any (nonempty) scheme $S$. 
This implies that in $\SH{S}$, there exists many $\eta$-periodic objects, that is, objects $M$ such that the map $\eta \colon \Gm \otimes M\to M$ is an equivalence.

In this short note, we observe that this discrepancy between motivic homotopy theory and classical homotopy theory disappears if one works in the étale local stable $\mathbb{A}^1$-homotopy category $\SHet{S}$
(see e.g.\ \cite[§5]{Bachmann2021etalerigidity} for a definition, in the étale setting 
we will always work with hypersheaves).
Our main result is the following:
\begin{introthm} [\cref{thm:eta-complete}]
    Let $S$ be a scheme. Then for $X\in\SHet{S}$, the object $X[\eta^{-1}]$ is zero (that is, $\eta$ is weakly nilpotent). In particular, every object of $\SHet{S}$ is $\eta$-complete, and $\eta$ acts nilpotently on any compact object of $\SHet{S}$.
\end{introthm}
A corollary of this result is the following:
\begin{introcor}[{\cref{lem:etale-descent-already-eta-complete}}]
    Let $S$ be a scheme.
    The étale sheafification functor $L_{\et} \colon \SH{S} \to \SHet{S}$
    factors canonically over $\SH{S}^{\wedge}_{\eta}$. In particular, any object of $\SH{S}$ that satisfies étale descent
    is already $\eta$-complete.
\end{introcor}

In good cases, we can compute the index of nilpotence of $\eta$ in $\SHet{S}$. For example, if $k$ is an algebraically closed field, we show that in $\SHet{k}$ we have $\eta^4=0$. More generally:
\begin{introthm}[{\cref{cor:exists-Galois-cover-eta4-zero,prop:eta4-is-zero-aic}}]
    Let $S$ be any scheme. 
    Then there exists a finite faithfully flat map $S' \to S$ such that $\eta^4$ is null in $\SHet{S'}$.

    If there exists a map $f \colon S \to \Spec(k)$ where $k$ is a field with $\mathrm{cd}_2(k) \leqslant 1$ and $\sup_{p \in \mathbb{P}} \mathrm{cd}_p(k) < \infty$ (e.g., any scheme defined over a finite field or an algebraically closed field),
    then $\eta^4$ is already null in $\SHet{S}$.
\end{introthm}

In fact, we conjecture the following:
\begin{conj*}[{\cref{conj-eta-4}}]
    For any scheme $S$ we have $\eta^4 \cong 0$ in $\SHet{S}$.
\end{conj*}

\begin{rmk*}
    Since $f^* \eta^4 \cong \eta^4$ where $f \colon S \to \Spec(\Z)$ 
    is the unique morphism,
    it is of course enough to show that $\eta^4 \cong 0$ in $\SHet{\Z}$.
\end{rmk*}

On the other hand, we know that on almost all schemes that $\eta^3$ is not null:
\begin{introthm}[{\cref{lem:eta3-not-null}}]
    Let $S$ be a scheme which has a point of characteristic not $2$.
    Then $\eta^3$ is not null in $\SHet{S}$.
\end{introthm}

\begin{rmk*}
    If $S$ is a scheme where all points are of characteristic $2$,
    then $\eta$ (and in particular $\eta^3$) is null in $\SHet{S}$ by \cref{lem:eta-null-in-char-2}.
\end{rmk*}
Using similar methods we can show that the étale motivic cohomology spectrum detects nilpotence, which is the étale version 
of May's Nilpotence conjecture \cite[Theorem A]{zbMATH06525830}:
\begin{introthm}[\cref{thm:detects-nilpotence,lem:detects-nilpotence-maps}]
    Let $S$ be an étale locally étale bounded scheme (\emph{e.g.} $S$ is of finite type over $\Z$ or a field of finite virtual cohomological dimension). 
    Let $R\in\CAlg{\SHet{S}}$, $p,q\in\Z$ and $\nu\colon \S^{p,q}\otimes R\to R$. 
    Then $\nu\otimes H\Z_{\et}$ is weakly nilpotent if and only if $\nu$ is weakly nilpotent.
\end{introthm}
As a corollary, we obtain a version of Nishida's Nilpotence Theorem \cite{zbMATH03411885} for étale motivic spectra:
\begin{introcor}[\cref{cor:Nishida,cor:Nishida+BS}]
    Let $S$ be an étale locally étale bounded scheme.
    Then any torsion map $\nu \colon \S^{p,q}\to \S$ in $\SHet{S}$ is weakly nilpotent.
    If $S$ satisfies the Beilinson--Soulé vanishing conjecture, cf.\ \cref{defn:BS}, and either $p > 0$ or $q > 0$, one can remove the assumption that $\nu$ is torsion. 
    This applies if $S$ is the spectrum of a number field or a finite field.
\end{introcor}

\subsubsection*{Acknowledgments.}
We thank Tom Bachmann for pointing out a mistake in an earlier version of this article,
and Julie Bannwart for helpful feedback on a draft. We thank Joseph Ayoub and Lucas Gerth for taking the time of answering questions on rigid geometry.
We thank Dan Isaksen for questions that led to \cref{sect:detect-nilpotence}.
We also want to thank ``Société Nationale des Chemins de fer Français'' 
and ``Deutsche Bahn'', whose combined efforts made it necessary for the first author to stay 
two unexpected days at the second author's home, during which parts of this article were written.

Klaus Mattis acknowledges support by the Deutsche Forschungsgemeinschaft
(DFG, German Research Foundation) through the Collaborative Research Centre TRR 326 \emph{Geometry and 
Arithmetic of Uniformized Structures}, project number 444845124.

Swann Tubach acknowledges support by the ÉNS de Lyon and the European Research Council (ERC) under
the European Union’s Horizon 2020 research and innovation programme ``EMOTIVE'' (grant agreement
no. 101170066).
\section{Recollections on completions and periodizations}
In this section, let $\Cat E$ be a stable presentably symmetric monoidal category with unit $\S$, and 
$\nu \colon T \to \S$ be a map with $T$ tensor-invertible.
Consider the functor $(-) \sslash \nu$ that is given as the cofiber of $T \otimes - \xrightarrow{\nu} -$. 

\begin{defn}
    We say that a map $f \colon X \to Y$ is a $\nu$-equivalence, if $f \sslash \nu$ 
    is an equivalence. Write $(-)^\wedge_\nu \colon \Cat E \to \Cat E$ for the associated 
    Bousfield localization at $\nu$-equivalences, called \emph{$\nu$-completion}.
    We write $\Cat E\phantom{}^\wedge_\nu$ for the essential image of $(-)^\wedge_\nu$.
\end{defn}

\begin{defn}
    We say that an object $X \in \Cat E$ is \emph{$\nu$-periodic} if $X \sslash \nu = 0$
    (equivalently, $T \otimes X \xrightarrow{\nu} X$ is an equivalence).
    The $\nu$-periodization functor is the Bousfield localization 
    $(-)[\nu^{-1}] \colon \Cat E \to \Cat E$ with essential image the subcategory of $\nu$-periodic objects.
    Since $\nu$-periodic objects are evidently closed under limits and colimits (for limits use that $T \otimes -$ preserves 
    them since $T$ is invertible), this localization exists by the adjoint functor theorem.
\end{defn}

\begin{lem}
    The localization functor $(-)[\nu^{-1}]$ is smashing,
    i.e.\ for all $X \in \Cat E$ we have $X[\nu^{-1}] \cong \S[\nu^{-1}] \otimes X$.
\end{lem}
\begin{proof}
    By \cite[Lemma A.5.2]{annala2022motivic} it is enough to show 
    that for all $X, Y \in \Cat E$ such that $X$ is $\nu$-periodic,
    then so is $X \otimes Y$ and $\imap{}{Y}{X}$.
    Here, $\imap{}{-}{-}$ denotes the internal hom object 
    in $\Cat E$ (which exists by the adjoint functor theorem).
    But now both $\nu \otimes X \otimes Y$ 
    and $\nu \otimes \imap{}{Y}{X} \cong \imap{}{Y}{\nu \otimes X}$ 
    are equivalences, since already $\nu \otimes X$ is one.
\end{proof}

\begin{lem} \label{lem:nu-equiv-iff-fiber}
    Let $f \colon X \to Y$ be a map in $\Cat E$. Then $f$ is a $\nu$-equivalence if and only if $\Fib{f}$
    is $\nu$-periodic.
\end{lem}
\begin{proof}
    The map $f$ is a $\nu$-equivalence if and only if $f \sslash \nu$ is an equivalence,
    i.e., if and only if $0 = \Fib{f \sslash \nu} \cong \Fib{f} \sslash \nu$.
    But the latter is zero if and only if $\Fib{f}$ is $\nu$-periodic.
\end{proof}

We now try to describe the $\nu$-periodization functor explicitly.
For this, recall the following definition:
\begin{defn}
    Let $X \in \Cat E$.
    We define the \emph{mapping telescope} $M_{\nu}(X)$ as the filtered colimit 
    \begin{equation*}
        \colim{} X \xrightarrow{\nu} T^{\otimes -1} \otimes X \xrightarrow{\nu} T^{\otimes -2} \otimes X \to \cdots
    \end{equation*}
    Since the tensor product is compatible with colimits, we see that $M_{\nu}(X) \cong M_{\nu}(\S) \otimes X$.
\end{defn}

Now, in a variety of situations, the mapping telescope agrees with the $\nu$-periodization.
\begin{lem} \label{lem:telescope}
    Suppose that there exists a compactly generated presentably symmetric monoidal stable category $\Cat D$
    with unit $\widetilde{\S}$,
    and a symmetric monoidal left adjoint $L \colon \Cat D \to \Cat E$.
    Suppose moreover that there exists a map $\widetilde{\nu} \colon \widetilde{T} \to \widetilde{\S}$ in $\Cat D$
    with $\widetilde{T}$ tensor invertible, such that $L(\widetilde{\nu}) \simeq \nu$.

    Then for all $X \in \Cat E$ we have $M_{\nu}(X) \cong X[\nu^{-1}]$.
\end{lem}
\begin{proof}
    We have $M_{\nu}(X) \cong M_{\nu}(\S) \otimes X$ 
    and $X[\nu^{-1}] \cong \S[\nu^{-1}] \otimes X$ (the latter holds since the localization is smashing).
    Hence, it suffices to prove the result for $X = \S$. It is clear that $\S\to M_{\nu}(\S)$ is sent to an equivalence by the functor $(-)[\nu^{-1}]$. Thus, it suffices to prove that $M_{\nu}(\S)$ is $\nu$-periodic. For this, as 
     $L(M_{\widetilde{\nu}}(\widetilde{\S})) \cong M_{\nu} (L\widetilde{\S}) \cong M_{\nu}(\S)$, it suffices to prove the statement in $\Cat D$, which is compactly generated.
    Then the result is \cite[Lemma 17]{bachmann2018motivic}.
\end{proof}

\begin{rmk}
    The last lemma holds for example if $\Cat E$ is compactly generated.
\end{rmk}

\begin{defn}
    Let $X \in \Cat E$.
    We say that \emph{$\nu$ acts weakly nilpotent on $X$} if $M_{\nu}(X) \cong 0$,
    and that \emph{$\nu$ acts nilpotent on $X$} if $\nu^n \colon X \to T^{\otimes -n} \otimes X$ is null for some $n \ge 1$.

    Similarly, we say that \emph{$\nu$ is (weakly) nilpotent}
    if $\nu$ acts (weakly) nilpotent on $\S$.
\end{defn}

\begin{lem} \label{lem:nu:cmpt-nilpotent}
    Let $X \in \Cat E$ an object.
    If $\nu$ acts nilpotent on $X$, then it acts weakly nilpotent on $X$.
    The converse holds if $X$ is compact.
\end{lem}
\begin{proof}
    The first statement holds trivially, so assume that $X$ is compact.
    Thus, we have a filtered colimit of abelian groups
    \begin{equation*}
    0 = \pi_0\Map{\Cat E}{X}{M_{\nu}(X)} \cong \colim{n} \pi_0\Map{\Cat E}{X}{T^{\otimes -n} \otimes X}.
    \end{equation*}
    In particular, the vanishing of the canonical map $X \to M_{\nu}(X)$ is witnessed 
    on some finite stage, whence we obtain $\nu^n \simeq 0$ for some $n \gg 0$.
\end{proof}

\section{Recollections on stable étale motivic homotopy theory}

\begin{defn} \label{def:etale-bdd}
    Let $S$ be a scheme. We say that $S$ is \emph{étale bounded} if 
    \[\mathrm{sup}_{x\in X, p \in \mathbb{P}} \mathrm{cd}_p(\kappa(x))<\infty,\]
    where $\mathrm{cd}_p(k)$ is the mod-$p$-Galois cohomological dimension of a field $k$, and $\mathbb{P}$ is the 
    set of all prime numbers.
    Similarly, we say that $S$ is \emph{étale locally étale bounded} if there exists 
    an étale cover $S' \to S$ such that $S'$ is étale bounded.
\end{defn}

\begin{exmpl} \label{lem:recap:Z}
    By \cite[Example 2.14]{Bachmann2021etalerigidity} the scheme $\Spec(\Z)$ is étale locally étale bounded,
    and the scheme $\Spec(\Z[i]) = \Spec(\Z[x] / (x^2 + 1))$ is étale bounded.
\end{exmpl}

\begin{lem} \label{lem:recap:etale-morphism-etale-finite}
    Let $X \to S$ be a morphism of finite type with $S$ quasi-compact.
    If $S$ is (étale locally) étale bounded, the same is true for $X$.
\end{lem}
\begin{proof}
    Suppose that $S$ is étale locally étale bounded. Choose an étale cover $S' \to S$ so that $S'$ is étale bounded.
    In particular, we get an étale cover $S' \times_S X \to X$, such that $S' \times_S X \to S'$ is of finite type.

    Hence, we may assume that $S$ is étale bounded, and we have to show that the same is true for $X$.
    For this, see \cite[Lemma 2.24]{mattis2025etale} (the extra assumptions 
    given in the reference that $S$ is of finite Krull dimension and that $X \to S$ is smooth are not necessary).
\end{proof}

\begin{lem} \label{lem:recap:compact}
    Let $S$ be a scheme that is étale locally étale bounded.
    Then $\SHet{S}$ is compactly generated.
    If $S$ is moreover étale bounded, 
    then $\Sigma^\infty_+ X \in \SHet{S}$ is compact for every qcqs $X \in \Sm{S}$.
\end{lem}
\begin{proof}
    This is \cite[Proposition 2.4.22]{AGV6FF} (see also Remark 2.4.23 of \emph{ibid}).
\end{proof}

\begin{thm}[Rigidity] \label{lem:recap:rigidity}
    Let $S$ be a scheme and $\ell$ a prime.
    Then there are canonical equivalences 
    \begin{equation*}
        \SHet{S}^\wedge_\ell \cong \SHet{S[1/\ell]}^\wedge_\ell \cong \ShvTopH{\et}{S[1/\ell],\Sp}^\wedge_\ell,    
    \end{equation*}
    where we write $S[1/\ell] \coloneqq S \times_{\Spec(\Z)} \Spec(\Z[1/\ell])$.
\end{thm}
\begin{proof}
    Consider the open immersion $S[1/\ell] \to S$,
    with closed reduced complement $Z$ of characteristic $\ell$.
    By \cite[Corollaire 4.5.4]{AyoubThesis} there is a recollement $\SHet{S[1/\ell]}^\wedge_\ell \to \SHet{S}^\wedge_\ell \to \SHet{Z}^\wedge_\ell$
    (note that Ayoub implicitly fixes a topology, which is allowed to be the étale topology,
    see the beginning of \cite[Section 4.5]{AyoubThesis}; the fact that Ayoub's result implies that we have an $\infty$-categorical recollement is classical, see \cite[Proposition 9.4.20]{robaloThese}).
    Moreover, $\SHet{Z}^\wedge_\ell = 0$ by \cite[Theorem A.1]{bachmann2021remarksetalemotivicstable}.
    This implies the first equivalence.
    The second equivalence in this generality is \cite[Theorem 3.1]{bachmann2021remarksetalemotivicstable}.
\end{proof}

We finish this section with a recollection on the motivic cohomology spectrum.

\begin{recoll}
    Recall the slice filtration (\cite{VSlice}).
    Let $S$ be a scheme and let $t\in\N$ be a positive integer. Recall that the $\infty$-category $\SH{S}^{\mathrm{eff}}$ of effective motivic spectra consists of the smallest stable subcategory stable under colimits that contains 
    \[\{\S^{p,q}\otimes \Sigma^{\infty}X_+\mid \ p\in\Z,\ q\geqslant 0,\ X\in \mathrm{Sm}_S\}\] with $\mathrm{Sm}_S$ the category of smooth $S$-schemes. We say that an object $M\in\SH{S}$ is $t$-effective if $\S^{0,-t}\otimes M\in \SH{S}^{\mathrm{eff}}$. The $\infty$-category of $t$-effective motivic spectra is colocalizing, and we denote by $f_t$ the colocalization functor (\emph{i.e.} the composition $f_t = iR$ of the inclusion $i$ with its right adjoint $R$). The co-unit induces a map $f_t M\to M$ for each $M\in\SH{S}$, called the $t$-effective cover. They induce the slice filtration on $M$. 
    The $t$-th slice of $M$ is the cofiber $s_t(M)$ of the morphism $f_{t+1}M\to f_{t}M$. 
    
    We define $H\Z$ as the zero-slice $s_0(\S)$ of the motivic sphere spectrum (note that this is indeed an $\mathbb{E}_\infty$-ring object of $\SH{S}$ by multiplicativity of the slice filtration $(f_t)_t$ \cite[Sect. 13.4]{bachmann2020norms}, and because $\S$ is effective). By \cite[Theorems 9.0.3 and 10.5.1]{LevineConiveau}, if $S$ is the spectrum of a field, then $H\Z$ represents Voevodsky's motivic cohomology constructed out of Bloch's cycle complex, and agrees with the motivic Eilenberg--Mac Lane spectrum $H\Z$ constructed out of sheaves with transfers in \cite[Section 4.1]{hoyois2015hopkinsmorel}. By \cite[Theorem B.4]{bachmann2020norms}, if $S$ is essentially smooth over a Dedekind domain, then $H\Z$ is Spitzweck's motivic cohomology spectrum (\cite{zbMATH07015021}), which is stable under base change between Dedekind domains or their points (\cite[Section 9]{zbMATH07015021}), thus we have that the map $\pi^*H\Z\to H\Z$ is an equivalence, where $\pi\colon S\to\Spec(\Z)$ is the structural morphism. In fact, for any map of schemes $f\colon Y\to X$ of schemes, the map $f^*H\Z \to H\Z$ is an equivalence in $\SH{X}$ by \cite[Theorem 9.3]{bachmann2025mathbba1invariantmotiviccohomologyschemes}. 
\end{recoll}
\begin{notation}
    \label{nota:top}
    Let $\Cat C$ be a presentably symmetric monoidal $\infty$-category. We denote by $(-)^\mathrm{top}\colon \Sp\to \Cat C$ the unique symmetric monoidal left adjoint.
\end{notation}
\begin{lem}
    The ring object $H\Z\in\SH{\Spec(\Z)}$ is $\Z$-linear.
\end{lem}
\begin{proof}
    By $\Z$-linear, we mean that there exists a ring map  $H\Z^\mathrm{top}\to H\Z$ with $H\Z^\mathrm{top}$ the image of the topological Eilenberg--Mac Lane spectrum $H\Z\in\Sp$ by the functor introduced in \cref{nota:top}. First, note that $H\Z\otimes \Q$ always has a $\Z$-linear structure as it is even $\Q$-linear. Next, by \cite[Lemma 13.7]{bachmann2020norms}, we know that for each nonzero integer $m$, the ring $H\Z/m$ (the $\mathbb{E}_\infty$-ring structure has been constructed by Spitzweck in \cite{zbMATH07015021}) is in fact an object of the heart of the effective t-structure on effective motivic spectra: $H\Z/m\in \SH{\Spec(\Z)}^{\mathrm{eff}\heartsuit}$. As any Grothendieck abelian category is $\Z$-linear, this provides a $\Z$-linear structure $H\Z^\mathrm{top}\to H\Z/m$ on each $H\Z/m$, thus also at the limit, for each prime number $\ell$, a map $H\Z^\mathrm{top}\to H\Z^\wedge_\ell$. Because the $\Z$-linear structure on a rational object is unique ($\Q$ being idempotent), the square 
    \[\begin{tikzcd}
	{H\Z^\mathrm{top}} & {\prod_{\ell}H\Z^\wedge_\ell} \\
	{H\Z\otimes\Q} & {\Big(\prod_{\ell}H\Z^\wedge_\ell\Big)\otimes\Q}
	\arrow[from=1-1, to=1-2]
	\arrow[from=1-1, to=2-1]
	\arrow[from=1-2, to=2-2]
	\arrow[from=2-1, to=2-2]
\end{tikzcd}\] commutes, thus the fracture square of $H\Z$ (combine \cite[Corollary 7.3]{mattis2024fracture} with \cite[Corollary 3.2.2.5]{higheralgebra}) provides us with a map of rings $H\Z^\mathrm{top}\to H\Z$, finishing the proof.
\end{proof}

\begin{defn}
    Write $H\Z_{\et} \coloneqq L_{\et} H\Z \in \SHet{S}$ for the étale motivic cohomology spectrum.
    Since $L_{\et}$ and $H\Z$ are compatible with pullbacks, we see that $H\Z_{\et}$ is compatible with pullbacks.
\end{defn}

\begin{prop}
    \label{prop:LetHZ=HLetZ}
     \label{cor:over-Z-HZ=HZtop}
    In $\SHet{\Spec(\Z)}$, the natural map $H\Z^\mathrm{top}\to H\Z_{\et}$ is an equivalence. 
\end{prop}
\begin{proof}
    That such a map exists follows from the fact that $H\Z$ is $\Z$-linear.
    As both sides are compatible with the pullback functor along $f \colon \Spec(k)\to \Spec(\Z)$ where $k=\finfld{p}$ or $k=\Q$, and as this family of functors is conservative by \cite[Corollary 5.12]{Bachmann2021etalerigidity}, it suffices prove the result over a field. It suffices to prove that the map is an equivalence upon $- \otimes \Q$ or $- / p$ for each prime number $p$. Rationally, we have that $H\Z_{\et}\otimes \Q \simeq H_{\mathfrak{B}}$ is Beilinson's motivic cohomology spectrum by \cite[Corollary 16.1.7]{cisinski2019triangulated}, which turns out to be $H\Q^\mathrm{top}$ as a consequence of \cite[Theorem 16.2.18]{cisinski2019triangulated} (compare the two possible images of the unit under the right adjoint of $\SHet{S}\to \mathrm{DA}^{\et}(S,\Q)$ using the identification of the theorem). Mod $p$ for $p$ a prime number, it suffices to check that the map is an equivalence after applying the functors $\pi_i\map{\SHet{k}}{\Gm^{\otimes j}\otimes \Sigma^\infty X_+}{-}$ for $X$ smooth over $k$ and $i,j\in\Z$. 
    Using rigidity as in \cref{lem:recap:rigidity}, if $p=\mathrm{char}(k)$ both sides are zero. Otherwise, for $H\Z^\mathrm{top}/p$ we obtain $\mathrm{H}^{-i-j}(X_{\et},\mu_{p}^{\otimes -j})$, which is isomorphic to the right-hand side as a consequence of the proof of the Bloch-Kato conjecture (see e.g.\ \cite[Theorem 1.2 5.]{zbMATH02132004}).
\end{proof}
Since everything is pulled back from $\Spec{\Z}$, we immediately obtain the following:
\begin{cor} \label{prop:LetHZ=HLetZ-any-scheme}
    Let $S$ be a scheme. Then $H\Z$ is $\Z$-linear in $\SH{S}$ and the canonical map $H\Z^\mathrm{top}\to H\Z_{\et}$ is an equivalence.
\end{cor}

\section{Nilpotence of \texorpdfstring{$\eta$}{eta}}

\begin{defn}
Let $S$ be a scheme. The algebraic Hopf map over $S$ is the map 
\[\eta\colon\Gm\to \S \] in $\SH{S}$ obtained as the $\mathbb{P}^1$-desuspension of the quotient map 
$\mathbb{A}^2_S\setminus\{0\} \to \mathbb{P}^1_S$.
\end{defn}
\begin{rmk}
    By smooth base change, if $f\colon T\to S$ is any map of schemes, then $f^*\eta \simeq \eta$ in $\SH{T}$.
\end{rmk}
We immediately get the following equivalence between the $\eta$-periodization and the mapping telescope.
\begin{prop} \label{lem:telescope-for-eta}
    Let $S$ be a scheme and $X \in \SHet{S}$. Then $X[\eta^{-1}] \cong M_{\eta}(X)$.
\end{prop}
\begin{proof}
    By the last remark, for the canonical left adjoint $f^* \colon \SHet{\Spec(\Z)} \to \SHet{S}$,
    we have $f^* \eta \simeq \eta$.
    Hence, the result follows from \cref{lem:telescope},
    since $\SHet{\Spec(\Z)}$ is compactly generated by \cref{lem:recap:compact}.
\end{proof}
\begin{lem}
    \label{lem:eta-2-tors}
    Assume that $-1$ is a sum of squares on a scheme $S$. Then $\eta = 0$ in $\SH{S}[\frac{1}{2}]$.
\end{lem}
\begin{proof}
    If $-1$ is a sum of $n$ squares in $S$, then there is a map of rings 
    \[\Z[x_1,\dots,x_n]/(x_1^2+\dots x_n^2+1)\to \Cat O_S(S).\] 
    Thus, we may assume that $S$ is of finite type over $\Spec(\Z)$.
    By \cite[Lemma 16.2.3]{cisinski2019triangulated} (in \emph{loc.~cit.~}they invert all primes, but the proof works \emph{verbatim} with only $2$ inverted, see their \cite[Remark 16.2.12]{cisinski2019triangulated}), there is some idempotent element $\varepsilon \in \mathrm{End}_{\SH{S}[\frac{1}{2}]}(\S[\frac{1}{2}])$ such that $\eta = \varepsilon \eta$. 
    In particular, $\SH{S}[\frac{1}{2}] \simeq \SH{S}[\frac{1}{2}]^+\times \SH{S}[\frac{1}{2}]^-$ where the $+$ part (\emph{resp}.\ the $-$ part) consists of modules over $\mathrm{Im}\frac{1-\varepsilon}{2}$ (\emph{resp}.\ modules over $\mathrm{Im}\frac{1+\varepsilon}{2}$). 
    The image of $\eta$ in $\SH{S}[\frac{1}{2}]^+$ is zero, since $\frac{1-\epsilon}{2} \eta = \frac{\eta - \eta}{2} = 0$. Thus, it suffices to show that 
    $\SH{S}[\frac{1}{2}]^- = 0$. By \cite[Proposition 4.3.17]{cisinski2019triangulated} (that we may apply thanks to \cite[Proposition 2.5.11]{AGV6FF}) we may assume that $S = \Spec(k)$ is the spectrum of a field.

    Over a field, the splitting of $\SH{k}[\frac{1}{2}]$ is induced by a splitting of the
    endomorphisms of the unit $GW (k)[\frac{1}{2}] \cong \Z[\frac{1}{2}] \times W (k)[\frac{1}{2}]$, where $GW(k)$ and $W(k)$ 
    are the Grothendieck-Witt ring and the Witt ring of $k$ (see e.g.\ \cite[§2.7.3]{bachmann2020eta}). 
    Hence, it suffices to show that under our assumptions $W (k)$ has $2$-power torsion. 
    This is for example proven in \cite[Chapter 2, Theorem 7.1]{scharlau2012quadratic}.
\end{proof}

\begin{prop}
    \label{prop:eta4-is-zero-w-contractible}
    Let $S=\Spec(\overline{\mathbb{Z}})$ be the spectrum of the integral closure of $\Z$ in $\overline{\Q}$.
    The map $\eta^4 \colon \Gm^{\otimes 4} \to \S$ is null in $\SHet{S}$.
\end{prop}
\begin{proof}
    We begin with the observation that $S$ is étale bounded: 
    its residual fields are $\overline{\Q}$ and copies of $\overline{\mathbb{F}}_p$ for all prime numbers $p$, which,
    as they are algebraically closed, have étale cohomological dimension $0$.
    Consider the arithmetic fracture square 
    \begin{center}
        \begin{tikzcd}
            \S \ar[r] \ar[d] \cartsymb & 
              \S^\wedge_2 \ar[d] \\
             \S[1/2] \ar[r] & \S^\wedge_2[1/2]
        \end{tikzcd}
    \end{center} in $\SHet{S}$ (cf.\ e.g.\ \cite[Corollary 7.3]{mattis2024fracture}).
    Mapping into this from $\Gm^{\otimes 4}$ gives the following cartesian square of mapping spectra:
    \begin{center}
        \begin{tikzcd}
            \map{}{\Gm^{\otimes 4}}{\S} \ar[r] \ar[d] \cartsymb& \map{}{\Gm^{\otimes 4}}{\S^\wedge_2} \ar[d] \\
            \map{}{\Gm^{\otimes 4}}{\S[1/2]} \ar[r] & \map{}{\Gm^{\otimes 4}}{\S^\wedge_2[1/2]} \rlap{.}
        \end{tikzcd}
    \end{center}
    On homotopy groups we get the following (part of a) long exact sequence:
    \[\begin{tikzcd}
        {\pi_1(\map{}{\Gm^{\otimes 4}}{\S^\wedge_2[1/2]})}\ar[r]& {\pi_0(\map{}{\Gm^{\otimes 4}}{\S})}\arrow[r]  &{\pi_0(\map{}{\Gm^{\otimes 4}}{\S^\wedge_2}) \oplus \pi_0(\map{}{\Gm^{\otimes 4}}{\S[1/2]})}
    .\end{tikzcd}\]
    By \cref{lem:recap:rigidity}, we have
    $\SHet{S}^\wedge_2\simeq \SHet{S[1/2]}^\wedge_2\simeq \ShvTopH{\et}{S[1/2],\Sp}^\wedge_2$, and the equivalence sends 
    $(\Gm)^\wedge_2$ to the object $\Sigma\S^\wedge_2$: indeed by \cite[Theorem 6.5]{Bachmann2021etalerigidity} $(\Gm)^\wedge_2$ is equivalent to the twisting spectrum $\hat{\mathds{1}}_2(1)[1]$, and by \cite[Theorem 3.6]{Bachmann2021etalerigidity} this twisting spectrum is equivalent to $\S^\wedge_2[1]$ when $S$ has all $2$-power roots of unity. 
    Recall also that $\Gm^{\otimes 4}$ is compact in $\SHet{S}$ (cf.\ \cref{lem:recap:compact}). 
    This allows us to rewrite the above exact sequence as:
    \[\begin{tikzcd}[column sep=0.8em]
        {\pi_5(\mathrm{R}\Gamma(S[1/2]_{\et},\S^\wedge_2))[1/2]}\ar[r] & {\pi_0(\map{}{\Gm^{\otimes 4}}{\S})}\arrow[r] &{\pi_4(\mathrm{R}\Gamma(S[1/2]_{\et},\S^\wedge_2)) \oplus \pi_0(\map{}{\Gm^{\otimes 4}[1/2]}{\S[1/2]})}.
    \end{tikzcd}\]  
    Write $(f,g)$ for the image of $\eta^4$ under the right map.

    It suffices to show that $f = 0 = g$, and moreover that $\pi_5(\mathrm{R}\Gamma(S[1/2]_{\et},\S^\wedge_2)) \cong 0$. Note first that because $S$ has all roots of unity, the map $g$ vanishes by \cref{lem:eta-2-tors}. We compute 
    $\pi_i(\mathrm{R}\Gamma(S[1/2]_{\et},\S^\wedge_2))$ for $i > 0$. Because the étale cohomological dimension of $S[1/2]$ is zero, the descent spectral sequence (\cite[Proposition 2.13]{Clausen_2021}, that we may apply because étale hypersheaves of spectra on $S$ are indeed Postnikov complete by \cite[Lemma 2.16]{Bachmann2021etalerigidity}, using that $S$ is étale bounded) 
    \[E_2^{p,q}=\mathrm{H}^p(S[1/2]_{\et},\pi_{-q}((\S^{\mathrm{top}})^\wedge_2))\Rightarrow \pi_{-p-q}(\mathrm{R}\Gamma(S[1/2]_{\et},\S^\wedge_2))\] ensures that 
    \[\pi_i(\mathrm{R}\Gamma(S[1/2]_{\et},\S^\wedge_2))\simeq \mathrm{H}^0(S[1/2]_{\et},\pi_i((\S^{\mathrm{top}})^\wedge_2)),\] with $\S^{\mathrm{top}}\in\Sp$ the topological sphere spectrum.
    Now, for both $i = 4$ and $i = 5$ we have that $\pi_i((\S^{\mathrm{top}})^\wedge_2) \cong 0$ (see e.g.\ the table after \cite[Definition 1.1.6]{ravenelgreen}),
    which implies $\pi_i(\mathrm{R}\Gamma(S[1/2]_{\et},\S^\wedge_2)) \cong 0$.
    This finishes the proof.
\end{proof}

Using a similar technique, we also have the following:

\begin{prop} \label{lem:codim-1-eta-4=0}
    Let $k$ be a field with $\mathrm{cd}_2(k) \leqslant 1$ and $\sup_{p \in \mathbb{P}} \mathrm{cd}_p(k) < \infty$ (\emph{e.g.\ }a finite field, cf.\ \cite[Chapter II, §3.3 (a)]{serre1994GaloisCohomology}, or a separably closed field). Then in $\SHet{k}$, we have $\eta^4=0$.
\end{prop}
\begin{proof}
    First note that because the $2$-cohomological dimension of $k$ is finite, $-1$ is a sum of squares in $k$:
    Indeed, suppose not. Then $k$ is orderable, and the absolute Galois group of $k$ contains 
    an element of order $2$. Thus, the $2$-cohomological dimension is infinite, cf.\ \cite[Chapitre II, §4.1, Proposition 10']{serre1994GaloisCohomology}.
    We begin as in \cref{prop:eta4-is-zero-w-contractible} (note that $\Spec(k)$ is étale bounded by assumption): 
    there is a short exact sequence 
        \[\begin{tikzcd}[column sep=0.8em]
        {\pi_{5}(\mathrm{R}\Gamma(k_{\et},\hat{\mathds{1}}_2(-4)   ))[1/2]}\ar[r] & {\pi_0(\map{\SHet{k}}{\Gm^{\otimes 4}}{\S})}\arrow[r] &{\pi_4(\mathrm{R}\Gamma(k_{\et},\hat{\mathds{1}}_2(-4))) \oplus \pi_0(\map{}{\Gm^{\otimes 4}[1/2]}{\S[1/2]})\rlap{,}}
    \end{tikzcd}\]  
    and we denote by $(f,g)$ the image of $\eta^4$ by the right map. By \cref{lem:eta-2-tors} we know that $g=0$. 
    We will now show that $f = 0$ be showing that the whole group $\pi_4(\mathrm{R}\Gamma(k_{\et},\hat{\mathds{1}}_2(-4)))$ vanishes.
    First, we need the following fact: $\pi_k (\hat{\mathds{1}}_2(-4))$ is $2$-power torsion for all $k > 0$ and vanishes if $k = 4$ or $k = 5$.
    For this, consider the t-exact conservative stalk functor $\rho^* \colon \ShvTopH{\et}{k_{\et}, \Sp} \to \ShvTopH{\et}{(k^\mathrm{sep})_{\et}, \Sp} = \Sp$.
    Note that $\rho^*(\hat{\mathds{1}}_2(-4)) \cong \hat{\mathds{1}}_2(-4) \cong \hat{\mathds{1}}_2$, the ($2$-completed) sphere spectrum,
    where we used \cref{lem:recap:rigidity} and \cite[Theorem 3.6 (2) and (3)]{Bachmann2021etalerigidity} (note that we do not need to re-$2$-complete after pulling back along $\rho$ because the identification can be made in the $\infty$-category of proétale sheaves, where $\rho^*$ commutes with limits as it is a slice).
    But now $\pi_k(\S)$ is torsion for every $k > 0$ by Serre's finiteness theorem \cite[Theorem 1.1.8]{ravenelgreen},
    and $\pi_4(\S) = \pi_5(\S) = 0$.

    Consider the descent spectral sequence (\cite[Proposition 2.13]{Clausen_2021}, again our sheaves are Postnikov complete by \cite[Lemma 2.16]{Bachmann2021etalerigidity})
    \begin{equation*}
        E_2^{p,q} = \pi_{-p} \mathrm{R}\Gamma(k_{\et}, \pi_{-q} (\hat{\mathds{1}}_2(-4))) \Rightarrow \pi_{-p-q} \mathrm{R}\Gamma(k_{\et}, \hat{\mathds{1}}_2(-4)).
    \end{equation*}
    Hence, to see that $\pi_4 \mathrm{R}\Gamma(k_{\et}, \hat{\mathds{1}}_2(-4)) = 0$,
    it suffices to show (using that $k_{\et}$ is of $2$-cohomological dimension $\le 1$ 
    and that $\pi_{n}(\hat{\mathds{1}}_2(-4))$ is $2$-power torsion for all $n > 0$) that $\pi_i (\hat{\mathds{1}}_2(-4)) = 0$ for $i = 4$ and $i = 5$.
    This we have seen above.

    Hence, we see that $\eta^4$ comes from an element in $\pi_{5}(\mathrm{R}\Gamma(k_{\et},\hat{\mathds{1}}_2(-4)))[1/2]$.
    Since $\eta$ (and hence $\eta^4$) is $2$-torsion by \cref{lem:eta-2-tors}, we conclude that $\eta^4 = 0$.
\end{proof}

\begin{cor}
    \label{prop:eta4-is-zero-aic}
    Let $S$ be a scheme with a map to a field $k$ with $\mathrm{cd}_2(k) \leqslant 1$ and $\sup_{p \in \mathbb{P}} \mathrm{cd}_p(k) < \infty$
    (e.g.\ any scheme of equicharacteristic $p > 0$, or any scheme defined over an algebraically closed field).
    The map $\eta^4 \colon \Gm^{\otimes 4} \to \S$ is null in $\SHet{S}$.
\end{cor}
\begin{proof}
    Since $\eta^4$ pulls back to $\eta^4$ along the map $S \to \Spec(k)$,
    we may assume that $S = \Spec(k)$, in which case the result is \cref{lem:codim-1-eta-4=0}.
\end{proof}

\begin{rmk}
    \label{rmk:Z[i]-H2}
    One would hope that a similar proof shows that $\eta^4 = 0$ over $\Z[i]$.
    This does not quite work, as in the spectral sequence we get an additional nonzero term given by the étale cohomology group 
    $H^2_{\et}(\Spec(\Z[i]), \pi_6(\hat{\mathds{1}}_2(-4))) \neq 0$.
\end{rmk}
Nonetheless, we conjecture the following:
\begin{conj}\label{conj-eta-4}
    For any scheme $S$ we have $\eta^4 \cong 0$ in $\SHet{S}$.
\end{conj}

\begin{rmk}
    \label{rmk:long-after-conj}
    \begin{enumerate}
    \item If the conjecture holds for $S$, and $f \colon S' \to S$ is a morphism,
        then the conjecture also holds for $S'$: Indeed, $\eta^4 \simeq f^* \eta^4 \simeq 0$.
   \item As noted in \cref{rmk:Z[i]-H2}, the obstruction for the argument in \cref{lem:codim-1-eta-4=0} to work for $\Q(i)$ lies in the apparition of a $\mathrm{H}^2_{\et}(\Q(i),\pi_6(\hat{\mathds{1}}_2(-4)))$. In fact, as the stalk of $\pi_6(\hat{\mathds{1}}_2(-4))$ is $\Z/2$ (thus $\pi_6(\hat{\mathds{1}}_2(-4))$ is the constant sheaf $\Z/2$ as $\mathrm{Aut}(\Z/2)= \{1\}$), and as $\pi_{5}(\mathrm{R}\Gamma(k_{\et},\hat{\mathds{1}}_2(-4)))[1/2]$ is always zero because it is the localization at $2$ of a $2$-torsion abelian group (this can again be read in the spectral sequence),  we obtain an isomorphism 
\[\pi_0(\map{\SHet{\Q(i)}}{\Gm^{\otimes 4}}{\S})\simeq \mathrm{Br}(\Q(i))[2]\] between the group in which lives $\eta^4$ and the $2$-torsion in the Brauer group of $\Q(i)$. By the Albert--Brauer--Hasse--Noether short exact sequence \cite[Theorem 4]{zbMATH00409910}, one may identify this right-hand side with 
\[\mathrm{Br}(\Q(i))[2]\simeq \mathrm{ker}(\bigoplus_p \mathrm{Br}(\Q_p(i))[2]\xrightarrow{\mathrm{inv}} \Q/\Z)\] and even better as each $\mathrm{Br}(\Q_p(i))[2]$ is isomorphic to $\Z/2$ (\cite[Chapter VI, §1.1, Theorem 1 and Corollary]{zbMATH03245597}), we see that we have an equivalence 
\[\pi_0(\map{\SHet{\Q(i)}}{\Gm^{\otimes 4}}{\S})\simeq \mathrm{ker}(\bigoplus_p \Z/2\xrightarrow{\mathrm{sum}} \Z/2).\]
The maps $\mathrm{Br}(\Q(i))[2]\ \to \mathrm{Br}(\Q_p(i))[2]$ are given by restriction, thus by the functoriality of the spectral sequence we used, we see that \emph{there is only finitely many prime numbers $p$ such that $\eta^4$ is nonzero in $\SHet{\Q_p(i)}$}. Moreover, that number of primes is even. This also implies that to prove \cref{conj-eta-4} for $\Spec(\Q(i))$ it suffices to do it for $\Spec(\Q_p(i))$ for all prime numbers. See \cref{rmk:eta4-zero-Rig} for an additional result in this direction.
\end{enumerate}
\end{rmk}

Using the result about $\Spec(\overline{\Z})$, we get the following weak version of the conjecture:

\begin{cor}
    \label{cor:exists-Galois-cover-eta4-zero}
    Let $S$ be a scheme. Then there exists a finite faithfully flat map $X\to S$ such that $\eta^4=0$ in $\SHet{X}$.
\end{cor}
\begin{proof}
    We claim that it is enough to show this for $\Spec(\Z)$.
    Indeed, suppose we have found a finite faithfully flat map $X \to \Spec(\Z)$ such that $\eta^4=0$ in $\SHet{X}$.
    Then consider the pulled back finite faithfully flat map $X \times S \to S$. Since $\eta^4$ in $\SHet{X \times S}$ 
    is pulled back from $\SHet{X}$, it vanishes.

    By \cref{prop:eta4-is-zero-w-contractible}, the map $\widetilde{X}\coloneq\Spec(\overline{\Z})\to\Spec(\Z)$ is such that $\eta^4=0$ in $\SHet{\widetilde{X}}$. 
    Note that $\widetilde{X}\to \Spec(\Z)$ factors through $\Spec(\Z[i])$. 
    In particular, we may write $\overline{\Z}$ as a filtered colimit of finite algebras $A_\alpha$ over $\Z[i]$, 
    which are thus all étale bounded by \cref{lem:recap:etale-morphism-etale-finite,lem:recap:Z}.

    Hence, if we denote by $X_\alpha = \Spec(A_\alpha)$, the categories $\SHet{X_\alpha}$ and $\SHet{\widetilde{X}}$ are compactly generated, with $\Gm$ and $\S$ are compact in them, cf.\ \cref{lem:recap:compact}.
    Moreover, by continuity (\cite[Proposition 2.5.11]{AGV6FF}), the map \[\colimil{\alpha} \pi_0(\mathrm{map}_{\SHet{X_\alpha}}(\Gm^{\otimes 4},\S)) \to \pi_0(\mathrm{map}_{\SHet{\widetilde{X}}}(\Gm^{\otimes 4},\S))\] 
    is an isomorphism of abelian groups (see \cite[A.10]{zbMATH07329546}). As $\eta^4$ vanishes in the colimit, there exists a finite level, say $X_{\alpha_0}$, such that $\eta^4=0$ in $\SHet{X_{\alpha_0}}$. This finishes the proof.
\end{proof}

\begin{cor}
    \label{cor:eta-period-zero}
    The Hopf map $\eta$ is weakly nilpotent in $\SHet{\Spec(\Z)}$.
\end{cor}
\begin{proof}
    Since $\Spec(\Z)$ is étale locally étale bounded, cf.\ \cref{lem:recap:Z},
    pulling back along the geometric points of $\Spec(\Z)$ is conservative, cf.\ \cite[Corollary 5.12]{Bachmann2021etalerigidity}.
    As for any map of schemes $f \colon S \to T$ we have $f^*\S[\eta^{-1}] \simeq \S[\eta^{-1}]$ (since $f^* \eta \cong \eta$, $f^*$ commutes with colimits and the $\eta$-periodization is given by the mapping telescope, cf.\ \cref{lem:telescope-for-eta}),
    the claim follows from the version for algebraically closed fields, cf.\ \cref{prop:eta4-is-zero-aic}.
\end{proof}

With the above we can prove:
\begin{thm}
    \label{thm:eta-complete}
    Let $X$ be a scheme. Then for any $M\in\SHet{X}$, we have $M[\eta^{-1}] \cong 0$, and the map 
    $M\to M^\wedge_\eta$ to its $\eta$-completion is an equivalence. 
    In particular, if $M$ is a compact object in $\SHet{X}$, there exists an integer 
    $n$ such that $\eta^n\colon M\otimes \Gm^{\otimes n}\to M$ vanishes.
\end{thm}
\begin{proof}
    Recall that the $\eta$-periodization is given by the mapping telescope,
    cf.\ \cref{lem:telescope-for-eta}.
    For the first part, using that the fiber of the map 
    $M\to M^\wedge_\eta$ is $\eta$-periodic (cf.\ \cref{lem:nu-equiv-iff-fiber}), we see that it suffices to check that any $\eta$-periodic object vanishes. 
    If $N\in\SHet{X}$ is an $\eta$-periodic object, 
    we have $N = N\otimes \S[\eta^{-1}]$, and denoting
    by $f\colon X\to \Spec(\Z)$ the structural morphism, 
    we have $\S[\eta^{-1}]\simeq f^*\S[\eta^{-1}]=0$, thus $N=0$.
    The second part now follows from \cref{lem:nu:cmpt-nilpotent}.
\end{proof}

\begin{cor} \label{lem:etale-descent-already-eta-complete}
    Let $S$ be a scheme.
    The étale sheafification functor $L_{\et} \colon \SH{S} \to \SHet{S}$
    factors canonically over $\SH{S}^{\wedge}_{\eta}$. In particular, any object of $\SH{S}$ that satisfies étale descent
    is already $\eta$-complete.
\end{cor}
\begin{proof}
    We have to see that $L_{\et}$ inverts every morphism $f \colon E \to F$ 
    such that $f \sslash \eta$ is an equivalence.
    We know from \cref{lem:nu-equiv-iff-fiber} that $\Fib{f}$ is $\eta$-periodic,
    which implies that $L_{\et} \Fib{f} \cong 0$, hence $L_{\et} f$ is an equivalence.
\end{proof}

\section{Nonvanishing of \texorpdfstring{$\eta^3$}{eta3}}

We proved in \cref{thm:eta-complete} that for every scheme $X$ and any object $M\in\SHet{X}$, the map $M\to M^{\wedge}_\eta$ is an equivalence. 
Moreover, by \cref{prop:eta4-is-zero-aic}, we know that for any scheme $S$ defined over a field $k$ of small étale cohomological dimension, we have $\eta^4 = 0$ in $\SHet{S}$. 
Even better, \cref{cor:exists-Galois-cover-eta4-zero}, we know that if $X$ is any scheme,
there exists a finite faithfully flat map $Y\to X$ such that $\eta^4=0$ in $\SHet{Y}$. It is surprisingly hard to descent the homotopy witnessing that $\eta^4=0$ over $Y$ to $X$.
On the other hand, in topology it is true that $\eta^3$ is not null.
In this section we show that this still holds in $\SHet{S}$ for any scheme $S$ not of equicharacteristic $2$.

\begin{prop}
    \label{lem:eta3-not-null-char-zero}
    Let $S$ be a nonempty scheme of characteristic zero. Then in $\SHet{S}$, the map $\eta^3$ is not null. 
    In fact, its image in $\SHet{S}^\wedge_2$ is not null.
\end{prop}
\begin{proof}
    Indeed, let $x\colon \Spec(k)\to S$ be a geometric point of $S$. 
    It suffices to prove that $\eta^3$ is not null in $\Spec(k)$. If $k$ admits an embedding into the field $\C$ of the complex numbers, 
    because the Betti realisation of $\eta$ is the topological Hopf map (note that the Betti realisation does factor through $\SHet{k}$ as in \cite[Definition 1.2.5]{ayoubAnabelianPresentationMotivic2022}, where $\SHet{S}$ is denoted by $\mathrm{MSh}(S,\S)$), we see that $\eta^3\neq 0$ (cf.\ \cite[Chapter V, Equation 5.5 and Proposition 5.6]{toda1962composition}).
     Even better, we have that $\eta^3\neq 0$ after $2$-completion (as $\pi_3(\S_\mathrm{top})$ is $2$-torsion).
    Unfortunately, it might be possible that $k$ is too big to be embedded in $\C$. In this case we work as follows:
    we have a map $f\colon\Spec(k)\to\Spec(\overline{\Q})$, which induces an equivalence on étale sheaves of spectra, thus in particular on $2$-completed étale sheaves of spectra:
    \[\Sp^\wedge_2\simeq \ShvTopH{\et}{\overline{\Q},\Sp}^\wedge_2\simeq \ShvTopH{\et}{k,\Sp}^\wedge_2.\]
    Moreover, we have the following commutative diagram
    \[\begin{tikzcd}
	{\SHet{k}} & {\SHet{\overline{\Q}}} & \Sp \\
	{\ShvTopH{\et}{k,\Sp}^\wedge_2} & {\ShvTopH{\et}{\overline{\Q},\Sp}^\wedge_2} & {\Sp^\wedge_2}
	\arrow["{\rho_2}"', from=1-1, to=2-1]
	\arrow["{f^*}"', from=1-2, to=1-1]
	\arrow["{\rho_{\mathrm{B}}}", from=1-2, to=1-3]
	\arrow["{\rho_2}"', from=1-2, to=2-2]
	\arrow["{(-)^\wedge_2}", from=1-3, to=2-3]
	\arrow["{f^*}", "\simeq"', from=2-2, to=2-1]
	\arrow["\simeq"', from=2-2, to=2-3]
\end{tikzcd}\] where $\rho_\mathrm{B}$ is the Betti realisation and $\rho_2$ is the $2$-adic étale realisation
(see the proof of \cite[Proposition 6.10]{ayoubWeil}). In particular, we see that:
\[(f^*)^{-1}(\rho_2(\eta^3_{\Spec(k)}))\simeq \rho_2(\eta^3_{\Spec(\overline{\Q})})\simeq \rho_\mathrm{B}(\eta^3_{\Spec(\overline{\Q})})^\wedge_2\neq 0,\] 
where we know that $\rho_\mathrm{B}(\eta^3_{\Spec(\overline{\Q})})^\wedge_2\neq 0$ 
by the first part of the proof.
This proves that $\eta^3$ is not null in $\SHet{k}$.
\end{proof}

\begin{thm} \label{lem:eta3-not-null}
    Let $S$ be nonempty scheme which is not of equicharacteristic $2$. Then in $\SHet{S}$, the map $\eta^3$ is not null.
\end{thm}
\begin{proof}
    We will use rigid analytic étale motives. As in the proof of \cref{lem:eta3-not-null-char-zero}, it suffices to prove that $\eta^3\neq 0$ in $\SHet{k}$ for $k$ an algebraically closed field. 
    If $k$ is of characteristic zero, this is \cref{lem:eta3-not-null-char-zero}. Thus, we assume that $\mathrm{char}(k)=p$ for some 
    prime number $p\neq 2$. Let $K'$ be the fraction field of the ring $W(k)$ of $p$-typical Witt vectors in $k$, and let $K$ be the completion of it algebraic closure. 
    As in the proof of \cite[Proposition 6.7]{ayoubWeil}, there is a commutative diagram of symmetric monoidal functors (see \cref{recoll:Rig-Formal-motives})
    \[\begin{tikzcd}
	{\SHet{k}} & {\mathrm{Rig}\mathcal{SH}_{\et}(K)} & {\SHet{K}} \\
	& {\Sp^\wedge_2}\rlap{.}
	\arrow["{\xi^*}", from=1-1, to=1-2]
	\arrow["{\rho_2}"', from=1-1, to=2-2]
	\arrow["{\rho_2}"', from=1-2, to=2-2]
	\arrow["{\mathrm{Rig}^*}"', from=1-3, to=1-2]
	\arrow["{\rho_2}", from=1-3, to=2-2]
\end{tikzcd}\]
We have that $\xi^*\eta_{\Spec(k)} \simeq \mathrm{Rig}^*\eta_{\Spec(K)}$. Indeed, this is shown in \cref{ex:P1A2} below.
In particular, we see that 
$\rho_2(\eta_{\Spec(k)})\simeq \rho_2(\eta_{\Spec(K)})$. The third power of the latter is not null by \cref{lem:eta3-not-null-char-zero}. This finishes the proof.
\end{proof}
\begin{rmk} \label{lem:eta-null-in-char-2}
    If $S$ is of equicharacteristic $2$, then there is a map of schemes $S\to\Spec(\mathbb{F}_2)$, and as $\SHet{\mathbb{F}_2}\simeq \SHet{\mathbb{F}_2}[\frac{1}{2}]$ by \cite[Lemma A.1]{bachmann2021remarksetalemotivicstable}, we see that $\eta=0$ in $\SHet{S}$ by \cref{lem:eta-2-tors}.
\end{rmk}

\begin{recoll}
      \label{recoll:Rig-Formal-motives}
    Let $k$ an algebraically closed field of characteristic $p>0$.
    Let $K$ be the completion of the algebraic closure of the fraction field of the ring $W(k)$ of $p$-typical Witt vectors in $k$. This is an algebraically closed valued field (see the proof of \cite[Remark 1.4.1]{zbMATH06706255} that works in this generality) with ring of integers (elements of norm $\leqslant 1$) that we denote $K^\circ$, which has residual field $K^\circ/\mathfrak{m}\simeq k$ (by \cite[Chapter II, Proposition 4.3]{zbMATH01313469}).
   
    Denote by $\mathrm{SmFSch}_{\mathrm{Spf}(K^\circ)}$ the category of $\mathfrak{m}$-adic smooth formal schemes over $\mathrm{Spf}(K^\circ)$,
    and by $\mathrm{SmRig}_K$ the category of smooth rigid analytic varieties over $K$.
    From the two categories $\mathrm{SmFSch}_{\mathrm{Spf}(K^\circ)}$ and $\mathrm{SmRig}_K$ one can define, as in algebraic geometry, categories of formal motives $\mathrm{F}\SHet{\mathrm{Spf}( K^\circ)}$ and rigid analytic motives $\mathrm{Rig}\SHet{K}$, say with the étale topology, suitably modified \cite[Definitions 3.1.3 and 2.1.15]{AGV6FF}. 
    One can define three functors on these categories:
    \begin{enumerate}
        \item The \emph{special fiber functor} $(-)_\sigma \colon \mathrm{SmFSch}_{\mathrm{Spf}(K^\circ)} \to \mathrm{Sm}_k$ that associates to a formal scheme $\Cat X \in \mathrm{SmFSch}_{\mathrm{Spf}(K^\circ)}$ its restriction along $\Spec(k)\to\mathrm{Spf}(K^\circ)$,
            and its induced functor $\sigma^* \colon \mathrm{F}\SHet{\mathrm{Spf}( K^\circ)} \to \SHet{k}$ \cite[Notations 1.1.6 and 3.1.9]{AGV6FF}.
        \item The \emph{generic fiber functor} $(-)^\mathrm{rig} \colon \mathrm{SmFSch}_{\mathrm{Spf}(K^\circ)} \to \mathrm{SmRig}_K$ that goes to the category of smooth rigid analytic varieties over $K$, 
            and its induced functor $\xi^* \colon  \mathrm{F}\SHet{\mathrm{Spf}( K^\circ)} \to \mathrm{Rig}\SHet{K}$ \cite[Notations 1.1.8 and 3.1.12]{AGV6FF}.
        \item The \emph{analytification functor} $(-)^\mathrm{an} \colon \mathrm{Sm}_K \to \mathrm{SmRig}_K$
            that sends a variety to its rigid-analytic version,
            and its induced functor $\mathrm{Rig}^* \colon \SHet{K} \to \mathrm{Rig}\SHet{K}$, see \cite[Construction 1.1.15 and Remark 2.2.6]{AGV6FF}. 
    \end{enumerate}
    Moreover, witnessing the phenomenon that ``$\mathbb{A}^1$-invariant motives do not see thickenings'', the special fiber functor $\sigma^*$ is an equivalence 
    $\mathrm{F}\SHet{\mathrm{Spf}( K^\circ)}\simeq \SHet{k}$,
    see \cite[Theorem 3.1.10]{AGV6FF} for this. 
    To summarize, there is a commutative diagram 
    \[\begin{tikzcd}
	{\mathrm{Sm}_k} & {\mathrm{SmFSch}_{\mathrm{Spf}(K^\circ)}} & {\mathrm{SmRig}_K} & {\mathrm{Sm}_K} \\
	{\SHet{k}} & {\mathrm{F}\SHet{\mathrm{Spf}( K^\circ)}} & {\mathrm{Rig}\SHet{K}} & {\SHet{K}} \rlap{.}
	\arrow[from=1-1, to=2-1]
	\arrow["{(-)_\sigma}", from=1-2, to=1-1]
	\arrow["{(-)^\mathrm{rig}}"', from=1-2, to=1-3]
	\arrow[from=1-2, to=2-2]
	\arrow[from=1-3, to=2-3]
	\arrow["{(-)^\mathrm{an}}", from=1-4, to=1-3]
	\arrow[from=1-4, to=2-4]
	\arrow["\simeq", from=2-2, to=2-1]
	\arrow["{\xi^*}"', from=2-2, to=2-3]
	\arrow["{\mathrm{Rig}^*}", from=2-4, to=2-3]
\end{tikzcd}\]
Finally, note that $K^\circ$-scheme $X$ gives an $\mathfrak{m}$-adic formal scheme $\Cat X$ through the $\mathfrak{m}$-adic formal completion functor that sends $X$ to $\Cat X \coloneqq \widehat{X} = X\times_{\Spec(K^\circ)}\mathrm{Spf}(K^\circ)$, the fiber product being taken in the category of formal schemes. In particular, if $X$ is a smooth $K^\circ$-scheme, then under the equivalence 
$\SHet{k}\simeq \mathrm{F}\SHet{\mathrm{Spf}( K^\circ)}$ the motive $\Sigma^\infty_+X_k$ of the $k$-scheme $X\times_{\Spec(K^\circ)}\Spec(k)$ corresponds to $\Sigma^\infty_+\widehat{X}$.
\end{recoll}
\begin{lem}
    \label{ex:P1A2}
    In the setting of \cref{recoll:Rig-Formal-motives}, we have $\xi^*\eta_{\Spec(k)} \simeq \mathrm{Rig}^*\eta_{\Spec(K)}$.
\end{lem}
\begin{proof}
    Indeed, by \cite[Corollaire 1.3.5]{AyoubRigMot}, the map 
    \[ (\widehat{\mathbb{A}^2_{K^\circ}\setminus \{0\}})^\mathrm{rig}\to (\mathbb{A}^2_{K}\setminus\{0\})^\mathrm{an} \] is an equivalence in $\mathrm{Rig}\SH{K}$. As $(\widehat{\mathbb{P}^1_{K^\circ}})^\mathrm{rig}\to (\mathbb{P}^1_{K})^\mathrm{an}$ is already an equivalence as rigid analytic varieties (see \cite[Definition 2.2.11]{AGV6FF}), the commutative square 
\[\begin{tikzcd}
	{(\widehat{\mathbb{A}^2_{K^\circ}\setminus \{0\}})^\mathrm{rig}} & {(\mathbb{A}^2_{K}\setminus\{0\})^\mathrm{an}} \\
	{(\widehat{\mathbb{P}^1_{K^\circ}})^\mathrm{rig}} & {(\mathbb{P}^1_{K})^\mathrm{an}}
	\arrow[from=1-1, to=1-2]
	\arrow[from=1-1, to=2-1]
	\arrow[from=1-2, to=2-2]
	\arrow[from=2-1, to=2-2]
\end{tikzcd}\] induces an equivalence $\xi^*\eta_{\Spec(k)} \simeq \mathrm{Rig}^*\eta_{\Spec(K)}$ in $\mathrm{Rig}\SHet{K}$.
\end{proof}

\begin{rmk}
    \label{rmk:eta4-zero-Rig}
    The \cref{ex:P1A2} does not require any of the fields $k$ and $K$ to be algebraically closed. In particular, we see that this applies to $K=\Q_p(i)$, which is interesting in view of \cref{rmk:long-after-conj} (4). Using \cref{lem:codim-1-eta-4=0} and \cref{ex:P1A2}, we see that $\mathrm{Rig}^*(\eta_{\Spec\Q_p(i)}^4)$ is zero. 
\end{rmk}

\section{Detecting nilpotence}
\label{sect:detect-nilpotence}
In the main body of this paper, we focused on the Hopf map $\eta$. In this last section, we prove
the étale motivic analog of May's Nilpotence Conjecture, proven by Mathew--Naumann--Noel \cite{zbMATH06525830}.
As a corollary, we obtain the étale analogue of Nishida's Nilpotence Theorem \cite{zbMATH03411885} about 
the nilpotence of all positive degree self maps of the étale sphere.
\begin{thm}[May nilpotence conjecture for $H\Z_{\et}$]
    \label{thm:detects-nilpotence}
    Let $S$ be an étale locally étale bounded scheme and let $R\in\CAlg{\SHet{S}}$ be a motivic ring spectrum. Then $R\otimes H\Z_{\et}=0$ if and only if $R=0$.
\end{thm}
\begin{proof}
    We may assume that $S=\Spec(k)$ is the spectrum of a separably closed field by \cite[Corollary 5.12]{Bachmann2021etalerigidity}. 
    Recall from \cref{prop:LetHZ=HLetZ-any-scheme} that $H\Z_{\et} \cong H\Z^{\mathrm{top}}$.
    Let $R\in\CAlg{\SHet{k}}$ be such that $R \otimes H\Z^{\mathrm{top}}=0$. Then $R\otimes \Q \cong R \otimes H\Q^{\mathrm{top}} = 0$.
    We claim that also $R^\wedge_\ell=0$ for each prime number $\ell$. 
    If $\ell = \mathrm{char}(k)$, this follows since the whole $\ell$-completed 
    category $\SHet{k}^\wedge_\ell$ vanishes by \cite[Theorem A.1]{bachmann2021remarksetalemotivicstable}.
    If $\ell \neq \mathrm{char}(k)$, by rigidity, we can view $R^\wedge_\ell$ as an honest ring spectrum $A\in\CAlg{\Sp}$, and we may use May's nilpotence conjecture, \cite[Theorem A]{zbMATH06525830}. 
    In particular, it suffices to prove that $A \otimes H\Z^\mathrm{top}\in\Sp$ vanishes, 
    for which it is enough to show that vanishing of $A\otimes H\Q^\mathrm{top}$ and $A\otimes H\finfld{p}^{\mathrm{top}}$ for all primes $p$. 
    That $A\otimes H\Q^\mathrm{top}$ vanishes can be seen as follows: In $\SHet{k}$, because there is a map of algebras $R\otimes \Q\to R^\wedge_\ell\otimes \Q$ whose source is zero, we have that $R^\wedge_\ell\otimes\Q = 0$ where the rationalisation 
    is taken in $\SHet{k}$. Now, a straightforward game of adjunctions, using that both $\S\in\SHet{k}$ and $\S\in\Sp$ are compact, and rigidity, shows that 
    \[\map{\SHet{k}}{\S}{R^\wedge_\ell\otimes \Q}\simeq \map{\SHet{k}}{\S}{R^\wedge_\ell}\otimes\Q \simeq \map{\Sp^\wedge_\ell}{\S^\wedge_\ell}{A}\otimes\Q\simeq \map{\Sp}{\S}{A\otimes H\Q^\mathrm{top}},\] so that in spectra the commutative ring $A\otimes H\Q^\mathrm{top}$ vanishes.
    If $p\neq \ell$, then $p$ is invertible on $R^\wedge_\ell$, thus $A\otimes H\finfld{p}^\mathrm{top}=0$. For $p=\ell$, we have 
    \[0=(R\otimes H\Z^{\mathrm{top}})^\wedge_\ell\sslash\ell \simeq (R\otimes H\Z^{\mathrm{top}}) \sslash \ell \simeq R \otimes H\finfld{\ell}^\mathrm{top} \simeq A\otimes H\finfld{\ell}^\mathrm{top}.\] 
    We conclude as the family of functors $(-\otimes \Q,((-)^\wedge_\ell)_{\ell\text{ prime}})$ form a conservative family on $\SHet{k}$.
\end{proof}
\begin{rmk}
    The reader familiar with \cite{bachmann2019nilpotence} may find it odd that we do not need $R$ to be normed in \cref{thm:detects-nilpotence}. 
    In fact by \cite[Corollary C.13]{bachmann2020norms} any commutative algebra in $\SHet{S}$ is automatically normed, 
    so the above result is coherent with the need of norms in \cite{bachmann2019nilpotence}.
\end{rmk}

\begin{cor} \label{lem:detects-nilpotence-maps}
    Let $S$ be an étale locally étale bounded scheme. Let $R\in\CAlg{\SHet{S}}$, $p,q\in\Z$ and $\nu\colon \S^{p,q}\otimes R\to R$. 
    Then $\nu\otimes H\Z_{\et}$ is weakly nilpotent if and only if $\nu$ is weakly nilpotent.
\end{cor}
\begin{proof}
    By \cref{lem:recap:compact} we know that $\SHet{S}$ is compactly generated,
    and hence by e.g.\ \cite[Lemma 2.2]{aoki2025higherpresentablecategorieslimits} 
    the same is true for $\Mod{R}{\SHet{S}}$. Hence, we see that in $\Mod{R}{\SHet{S}}$ 
    the $\nu$-periodization is given by the mapping telescope, cf.\ \cref{lem:telescope}.
    Now, that $\nu\otimes H\Z_{\et}$ is weakly nilpotent means that the ring $(R\otimes H\Z_{\et})[\nu^{-1}] \simeq R[\nu^{-1}]\otimes H\Z_{\et}$ vanishes. By \cref{thm:detects-nilpotence}, this implies that $R[\nu^{-1}]=0$, thus that $\nu$ is weakly nilpotent.
\end{proof}

We conclude with an étale motivic version of Nishida's nilpotence theorem \cite{zbMATH03411885}.

\begin{cor} [Nishida nilpotence for $\SHet{S}$] \label{cor:Nishida}
    Let $S$ be an étale locally étale bounded scheme.
    Then any torsion map $\nu \colon \S^{p,q}\to \S$ in $\SHet{S}$, is weakly nilpotent
    (in particular, if $S$ is étale bounded, any such $\nu$ is nilpotent).
   
\end{cor}
\begin{proof}
    By pulling back to the geometric points of $S$ we may assume $S$ is the spectrum of a separably closed field:
    Indeed, this is jointly conservative by \cite[Corollary 5.12]{Bachmann2021etalerigidity}, and 
    preserves the mapping telescopes as they are given by colimits.
    Since $\S$ is compact, there is a nonzero integer $m$ such that $m\nu=0$.  Consider the co/fiber sequence $H\Z_{\et}\to H\Z_{\et}[1/m]\to H\Z_{\et}[1/m]/H\Z_{\et}$. We obtain an exact sequence 
    \[\pi_1(\map{\SHet{k}}{\S^{p,q}}{H\Z_{\et}[1/m]/H\Z_{\et}})\to \pi_0(\map{\mathrm{DA}^{\et}(k,\Z)}{\S^{p,q}\otimes H\Z_{\et}}{H\Z_{\et}})\to \pi_0(\map{\SHet{k}}{\S^{p,q}}{H\Z_{\et}[1/m]}).\]
    The image of $\nu\otimes H\Z_{\et}$ by the right map is zero, so we see that $\nu$ comes from an element in the left group which is isomorphic to 
    \[\pi_1(\map{\SHet{k}}{\S^{p,q}}{H\Z_{\et}[1/m]/H\Z_{\et}})\simeq \colimil{a} \mathrm{Ext}^{-(p+q+1)}_{\Z/a\Z}(\Z/a\Z,\Z/a\Z) \cong 0\] where the colimit ranges over integers $a$ such that if a prime number $\ell$ divides $a$, 
    it divides $m$ (this isomorphism is otained as follows: first, write $H\Z_{\et}[1/m]/H\Z_{\et}$ as the filtered colimit $\colimil{a} H(\Z/a\Z)_{\et}$ with $a$ as above, and use that the sphere $S^{p,q}$ is compact. Next using adjunction and \cref{prop:LetHZ=HLetZ}, the mapping spectrum is computed in the $\Z/a\Z$-linearisation of $\SHet{k}$, which is, by rigidity, $\mathrm{D}(\Z/a\Z)$). 
    This group vanishes since $p + q + 1 \neq 0$ (we can always replace $(p,q)$ by $(2p,2q)$ and $\nu$ by $\nu^2$) and $\Z/a\Z$ is a projective $\Z/a\Z$-module.
    Thus, we see that $\nu\otimes H\Z_{\et} = 0$ thus by \cref{lem:detects-nilpotence-maps}, the map $\nu$ is weakly nilpotent.
    If $S$ is étale bounded it is nilpotent by \cref{lem:recap:compact,lem:nu:cmpt-nilpotent}. 
\end{proof}

Recall the following:
\begin{defn} \label{defn:BS}
    A scheme $S$ \emph{satisfies the Beilinson--Soulé vanishing conjecture} if 
    any map $\S^{p,q}\otimes H\Q_{\et}\to H\Q_{\et}$ with $p>0$ is null. 
\end{defn}
The above conjecture is stated in \cite[Section 3 Be3]{Suslin2000} (that this integral version is equivalent to \cref{defn:BS} is proven for example in \cite[Lem. 24]{zbMATH02247521}).
\begin{cor}
    \label{cor:Nishida+BS}
    If $S$ satisfies the Beilinson--Soulé vanishing conjecture, one can remove the assumption in \cref{cor:Nishida} that $\nu$ is torsion, adding the assumption that either $p>0$ or $q>0$.
\end{cor}
\begin{proof}
  We know that there are no nonzero maps $\S^{p,q}\otimes H\Q_{\et}\to H\Q_{\et}$ for $p>0$, thus if $p>0$, we see that $\nu$ is torsion. If $q>0$, by \cite[Corollary 4.26 (1)]{hoyois2015hopkinsmorel} we can also conclude that $\nu\otimes H\Q_{\et}=0$ so that $\nu$ is torsion.   
\end{proof}

\begin{rmk}
    \begin{enumerate}
        \item \cref{cor:Nishida+BS} applies to $k$ a finite field \cite{zbMATH03394392}, a number field (\cite{zbMATH03373095,zbMATH03495395}), or any of their algebraic extensions (by continuity). It also applies to function fields of curves over finite fields by \cite{zbMATH03606688}.
        \item One can state a slightly more general version of \cref{cor:Nishida}: it applies over any scheme to a torsion map which is pulled back from an étale locally étale bounded scheme.
        \item Of course, if $p = q = 0$, then there are non-torsion maps that are non-nilpotent, e.g.\ multiplication by any nonzero integer $m \in \Z$.
    \end{enumerate}
\end{rmk}

\bibliographystyle{alpha}
{\footnotesize
\bibliography{bibliography}}

\end{document}